\documentclass[12pt, reqno]{amsart}
\usepackage{amsmath, amsthm, amscd, amsfonts, amssymb, graphicx, color, mathrsfs}
\usepackage[bookmarksnumbered, colorlinks, plainpages]{hyperref}

\newtheorem{theorem}{Theorem}[section]
\newtheorem{lemma}[theorem]{Lemma}

\newtheorem{corollary}[theorem]{Corollary}
\theoremstyle{definition}

\newtheorem{example}[theorem]{Example}

\theoremstyle{remark}
\newtheorem{remark}[theorem]{Remark}
\numberwithin{equation}{section}

\begin{document}
\setcounter{page}{1}

\title[Nuclear pseudo-differential operators in Besov spaces]{  Nuclear pseudo-differential operators in Besov spaces on compact Lie groups}

\author[D. Cardona]{Duv\'an Cardona}
\address{
  Duv\'an Cardona:
  \endgraf
  Department of Mathematics  
  \endgraf
  Universidad de los Andes
  \endgraf
  Bogot\'a
  \endgraf
  Colombia
  \endgraf
  {\it E-mail address} {\rm d.cardona@uniandes.edu.co;
duvanc306@gmail.com}
  }



\subjclass[2010]{}

\thanks{The author was supported by the Faculty of Sciences of the  Universidad de los Andes, Project: Operadores en grupos de Lie compactos, 2016-I. No new data was created or generated during the course of this research.
}
\begin{abstract}
In this work we establish the metric approximation property for Besov spaces defined on arbitrary compact Lie groups. As a consequence of this fact, we investigate trace formulae for nuclear Fourier multipliers on Besov spaces. Finally, we study the  $r$-nuclearity, the Grothendieck-Lidskii formula and the (nuclear) trace of  pseudo-differential operators in generalized H\"ormander classes acting on periodic Besov spaces. We will restrict our attention to pseudo-differential operators  with symbols  of limited regularity. 
\end{abstract} \maketitle

\section{Introduction}
It was reported by H. Feichtinger, H. F\"uhr and I. Pesenson in \cite{FE} (see also references therein) that there exist many real life problems in signal analysis and information theory
which would require non-euclidean models. These models include: spheres, projective spaces and
general compact manifolds, hyperboloids and general non-compact symmetric spaces, and finally various Lie groups. In connection with these spaces it is important to study approximation theory, space-frequency
localized frames, and Besov spaces, on compact and non-compact manifolds. Motivated by these facts, in this paper we prove the approximation property of Grothendieck for Besov spaces defined on  general compact Lie groups. This property is of geometric nature and has important consequences in the theory of nuclear operators on Banach spaces \cite{GRO}. Consequently, by using the aproximation property on Besov spaces we investigate the $r$-nuclearity of global pseudo-differential operators on such spaces. This is possible if we take under consideration  the formulation of Besov spaces reported by  E. Nursultanov, M. Ruzhansky and M. Tikhonov in \cite{RuzBesov}, in the context of matrix-valued (or full) symbols of global pseudo-differential operators developed by M. Ruzhansky and V. Turunen in terms of the representation theory of compact Lie groups \cite{Ruz}.\\ 

In order to formulate our work we precise some definitions as follows. Through the work of A. Grothendieck and others methods in spectral theory, the theory of nuclear operators on Banach spaces has attracted much interest in the literature during the last fifty years, due to its applications in the problem of distribution of eigenvalues.  Let us consider $E$ and $F$ be two Banach spaces and let $0<r\leq 1.$  Following A. Grothendieck \cite{GRO}, Chapter II, p. 3, a linear operator $T:E\rightarrow F$ is  $r$-nuclear, if
there exist  sequences $(e_n ')_n$ in $ E'$ (the dual space of $E$) and $(y_n)_n$ in $F$ such that
\begin{equation}\label{nuc}
Tf=\sum_n e_n'(f)y_n
\end{equation}
and
\begin{equation}\label{nuc2}
\sum_n \Vert e_n' \Vert^r_{E'}\Vert y_n \Vert^r_{F}<\infty.
\end{equation}
\noindent The class of $r-$nuclear operators is usually endowed with the quasi-norm
\begin{equation}
n_r(T):=\inf\left\{ \left\{\sum_n \Vert e_n' \Vert^r_{E'}\Vert y_n \Vert^r_{F}\right\}^{\frac{1}{r}}: T=\sum_n e_n'\otimes y_n \right\}
\end{equation}
\noindent and, if $r=1$, $n_1(\cdot)$ is a norm and we obtain the ideal of nuclear operators. When $E=F$ is a Hilbert space and $r=1$ the definition above agrees with the concept of  trace class operators. For the case of Hilbert spaces $H$, the set of $r$-nuclear operators agrees with the Schatten-von Neumann class of order $r$ (see \cite{P}).
\\

\noindent The purpose of this paper is thus the study of the $r$-nuclearity of global pseudo-differential operators defined on Besov spaces in compact Lie groups \cite{Ruz}, these operators can be defined as follows: let us assume that  
$G$ is a compact Lie group and denote by  $\widehat{G}$ its unitary dual, i.e. the set of equivalence classes of all strongly continuous irreducible unitary representations of $G$. If $T$ is a linear operator from $C^{\infty}(G)$ into $C^{\infty}(G)$ and  $\xi:G\rightarrow U(H_{\xi})$ denotes  an irreducible unitary representation, we can associate to $T$ a matrix-valued symbol $a(x,\xi)\in \mathbb{C}^{d_{\xi}\times d_\xi}$ (see \eqref{symbol}) satisfying
\begin{equation}
Tf(x)=\sum_{[\xi]\in \widehat{G}}d_{\xi}\text{Tr}[\xi(x)a(x,\xi)(\mathscr{F}f)(\xi)],
\end{equation}
where in the summations is understood that from each class $[\xi]$ we pick just one representative $\xi\in[\xi]$,  $d_{\xi}=\text{dim}(H_\xi)$ and $(\mathscr{F}f)(\xi)$ is the Fourier transform at $\xi:$
\begin{equation}(\mathscr{F}f)(\xi):=\widehat{f}(\xi)=\int_{G}f(x)\xi(x)^* dx\in\mathbb{C}^{d_{\xi}\times d_{\xi}}.
\end{equation}
We are interested in the problem of the (nuclear) trace and trace formulae for $r$-nuclear pseudo-differential operators acting on Besov spaces defined in compact Lie groups as in \cite{Ruz}.
There are several possibilities, concerning the conditions to impose on a symbol $a(x,\xi)$, in the attempt to establish the $r$-nuclearity of the corresponding operator $T_a$ on Lebesgue spaces defined in compact Lie groups. This problem was considered by J. Delgado and M. W. Wong (c.f. \cite{DW}) in the commutative case of the torus $\mathbb{T}^n$. To the best of our knowledge, this is the first work on the nuclearity and $\frac{2}{3}$-nuclearity of pseudo-differential operators on the torus. \\ 
\\It is a well known fact that the approximation property on a Banach space is required to define the nuclear trace \cite{P}. A Banach space $E$ is said to have the approximation property if for every compact subset $K$ of
$E$ and every  $\varepsilon>0$ there exists a finite bounded operator $B$ on $E$ such that
\begin{equation}\label{eqapr}
\Vert x-Bx\Vert<\varepsilon,\,\,\,\text{for all }x\in {K}.
\end{equation}
On such spaces, if $T:E\rightarrow E$ is nuclear, the (nuclear) trace is defined by
\begin{equation}
\text{Tr}(T)=\sum_{n}e_{n}'(y_n),
\end{equation}
where $T=\sum_n e_n'\otimes y_n$ is a representation of $T.$ If in the definition above $\Vert B\Vert\leq 1$, one says that $E$ has the metric approximation property. It is well known that every $L^{p}(\mu)$ space with $1\leq p<\infty$ satisfies the approximation property. However, there exist Banach spaces which does not satisfy the approximation property. A counterexample to the statement that every Banach space $E$ has the approximation property was given early by P. Enflo in \cite{Enflo}. Enflo shows that there exists a separable reflexive Banach space with a sequence $M_{n}$ of finite dimensional subspaces with $\dim(M_n)\rightarrow \infty,$ and a constant $c$ such that for every operator $T$ of finite rank, $\Vert T-I\Vert\geq 1-c\Vert T\Vert/\log (\dim M_n).$ We refer the reader to  \cite{anisca} for a work on subspaces of $l^{2}(X)$ without the approximation property. A remarkable result due to A. Grothendieck shows that for every $\frac{2}{3}$-nuclear
operator $T$ acting on a Banach space $E,$ the (nuclear) trace $\text{Tr}(T)$ is well defined, the system of its eigenvalues is absolutely summable and the (nuclear) trace  is equal to the sum of the eigenvalues (see \cite{GRO}, chapter II).\\
\\
The  $r$-nuclearity of  operators give rise to  results on the distribution of their eigenvalues (see \cite{KH}).  This fact and the notion of spectral trace has been crucial in the study of spectral properties of nuclear operators  arising in many different contexts and  applications, such as the heat kernel on compact manifolds, the Fox-Li operator in laser engineering and transfer operators in thermodynamics (see  \cite{Fox1,Fox2,Ter}.) \\
\\
In this paper, which is based on the recent formulation of Besov spaces $B^{w}_{p,q}(G)$ on compact Lie groups given in \cite{RuzBesov}, we prove that these spaces have the metric approximation property for $w\in \mathbb{R},$  $1\leq p<\infty$ and $1\leq q\leq \infty$.  Consequently, we derive a trace formula for $r$-nuclear operators in these spaces and, by using the compact Lie group structure of the $n$-dimensional torus $\mathbb{T}^n$, we prove some sufficient conditions for the $r$-nuclearity of periodic pseudo-differential operators on Besov spaces. The results are applied to study the corresponding trace formula of Grothendieck-Lidskii, which shows that the nuclear trace of these operators coincides with the spectral trace defined as the sum of eigenvalues. Similar results in the literature, on the $r$-nuclearity in $L^p$-spaces for pseudo-differential operators defined on compact Lie groups or on compact manifolds, can be found in the recent works of J. Delgado and M. Ruzhansky \cite{DW, DR,DR1, DR3} and references therein. The reference \cite{DRT} consider the $r$-nuclearity of operators on manifolds with boundary. Mapping properties of pseudo-differential operators in Besov spaces defined on compact Lie groups can be found in \cite{Besov}. \\
\\
There exist several recent works about the approximation property. In function spaces on euclidean domains as Lebesgue spaces with variable exponent, the space of functions of bounded variation, Sobolev spaces $W^{1,1},$  modulation spaces, Wiener-Amalgam spaces, and holomorphic functions on the disk, we refer the reader to \cite{BVap, brudni, DR4, DRB, DRB2, lancien,  RoWo}. Recent works on the approximation property for abstract Banach spaces can be found in \cite{achour,kim1,kim2,lee1,lee2,li1}. For a historical perspective on the approximation property  we refer the reader to  A. Pietsch \cite{P2}. \\ 
\\ 
We now describe the contents of the paper in more detail. In Section 2, Theorem \ref{mainapp}, we present the metric approximation property for Besov spaces on compact Lie groups and some results are proved with respect to the $r$-nuclearity of Fourier multipliers. In Section 3, Theorem \ref{t1} and results therein provide sufficient conditions for the $r$-nuclearity of pseudo-differential operators acting on Besov spaces on the torus. Finally, in sections 4 and 5, we establish trace formulae for $r$-nuclear periodic pseudo-differential operators on Besov spaces and specific periodic operators as negative powers of the Bessel potential and the heat kernel on the torus.
 
\section{The metric approximation property for Besov spaces and nuclearity of Fourier multipliers in Compact Lie groups}

In this section we prove the metric approximation property for Besov spaces and we use this notion with the goal of investigate the nuclear trace of operators on Besov spaces. For the analysis on compact Lie  groups,   we refer the reader to \cite{Ruz, Ruz22}. See also \cite{RuzFisch} for a concise review of the theory on compact Lie groups. For the proof of the approximation property we use the following lemma (see e.g. \cite{P}):
\begin{lemma}\label{lemadeaprox}
A Banach space $E$ satisfies the metric approximation property if, given $f_{1},f_{2},\cdots ,f_{m}\in E$
and $\varepsilon>0$ there exists an operator of finite rank $B$ on $E$ such that $\Vert B\Vert\leq 1$ and 
\begin{equation}
\Vert f_i-Bf_i \Vert<\varepsilon,\hspace{0.5cm}1\leq i\leq m.
\end{equation}
\end{lemma}
\noindent Let us consider a compact Lie group $G$ with unitary dual $\widehat{G}$ that is, the set of equivalence classes of all strongly continuous irreducible unitary representations of $G.$ We will equip $G$ with the Haar measure $\mu_{G}$ and, for simplicity, we will write $\int_{G}f(x)dx$ for $\int_{G}f d\mu_{G}$, $L^p(G)$ for $L^p(G,\mu_{G})$, etc. The following equalities follow from the Fourier transform on $G$
 $$  \widehat{f}(\xi)=\int_{G}f(x)\xi(x)^*dx,\,\,\,\,\,\,\,\,\,\,\,\,\,\,\,\, f(x)=\sum_{[\xi]\in \widehat{G}}d_{\xi}\text{Tr}(\xi(x)\widehat{f}(\xi)) ,$$
and the Peter-Weyl Theorem on $G$ implies the Plancherel identity on $L^2(G),$
$$ \Vert f \Vert_{L^2(G)}= \left(\sum_{[\xi]\in \widehat{G}}d_{\xi}\text{Tr}(\widehat{f}(\xi)\widehat{f}(\xi)^*) \right)^{\frac{1}{2}}=\Vert  \widehat{f}\Vert_{ L^2(\widehat{G} ) } .$$
\noindent Notice that, since $\Vert A \Vert_{HS}=\text{Tr}(AA^*)$, the term within the sum is the Hilbert-Schmidt norm of the matrix $A$. Any linear operator $T_{a}$ on $G$ mapping $C^{\infty}(G)$ into $\mathcal{D}'(G)$ gives rise to a {\em matrix-valued global (or full) symbol} $a(x,\xi)\in \mathbb{C}^{d_\xi \times d_\xi}$ given by
\begin{equation}\label{symbol}
a(x,\xi)=\xi(x)^{*}(A\xi)(x),
\end{equation}
which can be understood from the distributional viewpoint. Then it can be shown that the operator $A=T_a$ can be expressed in terms of such a symbol as 
\begin{equation}\label{mul}T_{a}f(x)=\sum_{[\xi]\in \widehat{G}}d_{\xi}\text{Tr}[\xi(x)a(x,\xi)\widehat{f}(\xi)]. 
\end{equation}
We introduce Sobolev and Besov spaces on compact Lie groups using the  Fourier transform on the group $G$ as follows.  There exists a non-negative real number $\lambda_{[\xi]}$ depending only on the equivalence class $[\xi]\in \hat{G},$ but not on the representation $\xi,$ such that $-\mathcal{L}_{G}\xi(x)=\lambda_{[\xi]}\xi(x)$, where $\mathcal{L}_{G}$ is the Laplacian on the group $G$ (in this case, defined as the Casimir element on $G$). If we denote by $\langle \xi\rangle$  the function $\langle \xi \rangle=(1+\lambda_{[\xi]})^{\frac{1}{2}}$,   for every $s\in\mathbb{R}$ the  Sobolev space $H^s(G)$ on the Lie group $G$ is  defined by the condition: $f\in H^s(G)$ if only if $\langle \xi \rangle^s\widehat{f}\in L^{2}(\widehat{G})$. 
The Sobolev space $H^{s}(G)$ is a Hilbert space endowed with the inner product $\langle f,g\rangle_{s}=\langle (I-\mathcal{L}_{G})^{\frac{s}{2}}f, (I-\mathcal{L}_{G})^{\frac{s}{2}}g\rangle_{L^{2}(G)}$, where, for every $s\in\mathbb{R}$, $(I-\mathcal{L}_{G})^{\frac{s}{2}}:H^r\rightarrow H^{r-s}$ is the bounded pseudo-differential operator with symbol $\langle \xi\rangle^{s}I_{\xi}$. 
Now, if  $w\in\mathbb{R},$ $0< q\leq\infty$ and $0<p\leq \infty,$ the Besov space $ B^w_{p,q}(G)$ is the set of measurable functions on $G$ satisfying
\begin{equation}\Vert f \Vert_{B^w_{p,q}}:=\left( \sum_{m=0}^{\infty} 2^{mwq}\Vert \sum_{2^m\leq \langle \xi\rangle< 2^{m+1}}  d_{\xi}\text{Tr}[\xi(x)\widehat{f}(\xi)]\Vert^q_{L^p(G)}\right)^{\frac{1}{q}}<\infty.
\end{equation}
If $q=\infty,$ $B^w_{p,\infty}(G)$ consists of those functions $f$ satisfying
\begin{equation}\Vert f \Vert_{B^w_{p,\infty}}:=\sup_{m\in\mathbb{N}} 2^{mw}\Vert \sum_{2^m\leq \langle \xi\rangle < 2^{m+1}}  d_{\xi}\text{Tr}[\xi(x)\widehat{f}(\xi)]\Vert_{L^p(G)}<\infty.
\end{equation} 
A recent work on Besov spaces defined on homogeneous compact manifolds can be found in \cite{RuzBesov}. In the following theorem we present the metric approximation property for Besov spaces defined on arbitrary compact Lie groups. 
\begin{theorem}\label{mainapp}
Let $G$ be a compact lie group. If $1\leq p<\infty,$ $1\leq q\leq \infty$ and $w\in \mathbb{R},$ then the Besov space $B^{w}_{p,q}(G)$ satisfies the metric approximation property. 
\end{theorem}
\begin{proof}
First, we will prove the metric approximation property for $B^{w}_{p,q}(G).$  If  $1\leq p,q<\infty,$ $B^{w}_{p,q}(G)$ is a Banach space. Let us consider $f_1,f_2,\cdots, f_{m}\in B^{w}_{p,q}(G).$ Then, by  definition of Besov norm
\begin{equation}
\Vert f_{i} \Vert^{q}_{B^w_{p,q}}=\sum_{s=0}^{\infty}2^{swq}\Vert \sum_{2^s\leq \langle \xi\rangle <2^{s+1}} d_{\xi}\mathrm{Tr}[\xi(x)\widehat{f_{i}}(\xi)] \Vert_{L^p}^q <\infty. 
\end{equation}
Let us consider the operator $T_{N}$ on $B^{w}_{p,q}(G)$ defined by
\begin{align*}
S_{N}f(x)&:=\sum_{\langle \xi\rangle\leq N}d_{\xi}\textrm{Tr}(\xi(x)\widehat{f}(\xi)) \\
&=\sum_{\langle \xi\rangle\leq N}\sum_{i,j=1}^{d_{\xi}}d_{\xi}\xi_{ij}(x)\widehat{f}(\xi)_{ji},
\end{align*}
where the summation is understood that from each class  $[\xi]$ we pick just one representative $\xi\in [\xi].$  Clearly, for every $N$ the operator $S_{N}$ has finite rank. Moreover, 
\begin{equation}
\text{Rank}(S_N)=\text{span}\{\xi_{i,j}(x):1\leq i,j\leq d_{\xi}, \langle \xi\rangle\leq N\}. 
\end{equation}
Now, let  $T_{N}=\Vert S_{N} \Vert^{-1}S_N.$ Clearly, $\Vert T_N\Vert=1$ for every $N.$  On the other hand, $T_{N}$ is a Fourier multiplier with matrix valued symbol $$\sigma_{N}(\xi)=\Vert S_N\Vert^{-1}\chi_{\{\xi:\langle \xi\rangle\leq N\}}I_{d_{\xi}}$$ where $I_{d_\xi}$ is the matrix identity  on $\mathbb{C}^{d_{\xi}\times d_{\xi}}.$ We observe that
\begin{align*}
\Vert f_{i}-T_{N}f_{i} \Vert^{q}_{B^w_{p,q}} &=\sum_{s=0}^{\infty}2^{swq}\Vert \sum_{2^s\leq \langle \xi\rangle <2^{s+1}} d_{\xi}\mathrm{Tr}[\xi(x)(\widehat{f_{i}}(\xi)-\sigma_{N}(\xi)\widehat{f}_{i}(\xi)] \Vert_{L^p}^q 
\end{align*}
Using the fact that $ S_N $ converges in the operator norm to the identity operator on $B^{w}_{p,q}$ we have $\lim_{N}\Vert S_{N}\Vert=1.$ Observing that   $ \lim_{N\rightarrow \infty} \chi_{\{\xi:\langle \xi\rangle\leq N\} }(\eta)=1  $, $[\eta]\in \widehat{G},$ and by using the convergence dominated theorem we obtain  
$$ \lim_{N\rightarrow \infty} \Vert f_{i}-T_{N}f_{i} \Vert^{q}_{B^w_{p,q}}=0. $$
So, if $\varepsilon>0$ is given, for every $i$ there exists $N_{i}$ such that, if $N\geq N_{i}$ then $$ \Vert f_{i}-T_{N}f_{i} \Vert_{B^w_{p,q}} <\varepsilon. $$
Thus, if $$ M=\max \{N_i:1\leq i\leq m\}$$ we have $\Vert T_{M}\Vert= 1$ and $ \Vert f_{i}-T_{M}f_{i} \Vert_{B^w_{p,q}} <\varepsilon $ for $1\leq i\leq m.$ So, by applying  Lemma \ref{lemadeaprox},  $B^{w}_{p,q}(G)$ has the metric approximation property for $w\in\mathbb{R},$ $1\leq p<\infty$ and $1\leq q \leq \infty.$ The result for $q=\infty$ has an analogous proof.
\end{proof}
Now, we investigate the nuclear trace of Fourier multipliers on compact Lie groups. The following theorem will be an useful tool in order to establish the $r$-nuclearity of operators in Besov spaces. (An analogous result on Sobolev spaces  has been proved in \cite{DR-1}, Theorem 3.11).
\begin{theorem}\label{equivalencia}
Let $G$ be a compact Lie group, $0<r\leq 1,$ $1\leq p<\infty$ and $1\leq q\leq \infty.$ Let us consider $T_{a}:C^{\infty}(G)\rightarrow C^{\infty}(G)$ be a Fourier multiplier  with matrix valued symbol $a(\xi).$ The following two announcement are equivalent.
\begin{itemize}
\item{$(1)$} $T_{a}$ extends to a $r$-nuclear operator from $B^{w_{0}}_{p,q}(G)$ into $B^{w_{0}}_{p,q}(G)$ for some  $w_0\in \mathbb{R}.$
\item{$(2)$} $T_{a}$ extends to a $r$-nuclear operator from $B^{w}_{p,q}(G)$ into $B^{w}_{p,q}(G)$ for all  $w\in \mathbb{R}.$
\end{itemize}
In this case, the nuclear trace of $T_{a}$ is independent of the index $w\in\mathbb{R}.$
\end{theorem}
\begin{proof}
It is clear that $(2)$ implies $(1).$ Now, we will prove that $(1)$ implies $(2)$. Let us assume that $w,w_0\in \mathbb{R},$ $w\neq w_0$ and $T_{a}:B^{w_{0}}_{p,q}(G)\rightarrow B^{w_{0}}_{p,q}(G)$ is $r-$nuclear. The operators $(I-\mathcal{L}_{G})^{\frac{w_0-w}{2}}:B^{w_0}(G)\rightarrow B^{w}(G)$ and $(I-\mathcal{L}_{G})^{\frac{w-w_0}{2}}:B^{w}_{p,q}(G)\rightarrow B^{w_{0}}_{p,q}(G)$ are both an isomorphism of Besov spaces. By considering that the class of $r$-nuclear operators is an ideal on the set of bounded operators, the following factorization of $T_{a}:$ $$T_{a}=(I-\mathcal{L}_{G})^{\frac{w_0-w}{2}}\circ T_a\circ (I-\mathcal{L}_{G})^{\frac{w-w_0}{2}}:B^{w}_{p,q}(G)\rightarrow B^{w}_{p,q}(G)$$ 
shows that $T_{a}$ from $B^{w}_{p,q}(G)$ into $B^{w}_{p,q}(G)$ is $r$-nuclear. In the factorization above we have used that $T_{a}$ is a Fourier multiplier. Now, we will prove that the spectral trace of $T_a$ is independent of $w.$ In fact, let us consider a nuclear representation of $T_{a}:B^{w_0}_{p,q}\rightarrow B^{w_0}_{p,q} $: 
$$T_{a}f=\sum_{n}e_{n}'(f)y_n$$
where $(e_{n}')_n\subset (B^{w_0}_{p,q})'$ and $(y_{n})_{n}\subset B^{w_0}_{p,q}$ are sequences satisfying
$$ \sum_{n} \Vert e_{n}\Vert^{r}_{(B^{w_0}_{p,q})'}\Vert y_n \Vert^{r}_{B^{w_0}_{p,q}} <\infty.$$
By Theorem \ref{mainapp} the space $B^{w_{0}}_{p,q}(G)$ has the approximation property and the nuclear trace of $T_{a}$ is well defined. Since
\begin{align*}T_{a}f &=(I-\mathcal{L}_{G})^{\frac{w_0-w}{2}}\circ T_a\circ (I-\mathcal{L}_{G})^{\frac{w-w_0}{2}}f\\
&=\sum_{n}[e_{n}'\circ (I-\mathcal{L}_{G})^{\frac{w-w_0}{2}}(f)](I-\mathcal{L}_{G})^{\frac{w_0-w}{2}}(y_n)
\end{align*}
and 
\begin{align*}
\sum_{n} & \Vert e_{n}'\circ (I-\mathcal{L}_{G})^{\frac{w-w_0}{2}}  \Vert^{r}_{(B^{w}_{p,q})'}\Vert (I-\mathcal{L}_{G})^{\frac{w_0-w}{2}}(y_n)\Vert^r_{B^{w}_{p,q}} \\
&\leq \sum_{n} \Vert (I-\mathcal{L}_{G})^{\frac{w-w_0}{2}}\Vert_{B(B^{w}_{p,q},B^{w_0}_{p,q})} \Vert (I-\mathcal{L}_{G})^{\frac{w_0-w}{2}}\Vert_{B(B^{w_0}_{p,q},B^{w}_{p,q})} \Vert e_{n}\Vert^{r}_{(B^{w_0}_{p,q})'}\Vert y_n \Vert^{r}_{B^{w_0}_{p,q}}\\
&= \Vert (I-\mathcal{L}_{G})^{\frac{w-w_0}{2}}\Vert_{B(B^{w}_{p,q},B^{w_0}_{p,q})} \Vert (I-\mathcal{L}_{G})^{\frac{w_0-w}{2}}\Vert_{B(B^{w_0}_{p,q},B^{w}_{p,q})} \sum_{n}  \Vert e_{n}\Vert^{r}_{(B^{w_0}_{p,q})'}\Vert y_n \Vert^{r}_{B^{w_0}_{p,q}}\\
&<\infty ,
\end{align*}
we obtain that the nuclear trace of $T_{a}:B^{w}_{p,q}\rightarrow B^{w}_{p,q} $ is given by
\begin{align*}
\mathrm{Tr}(T_a)&=\sum_{n}e_{n}'\circ (I-\mathcal{L}_{G})^{\frac{w-w_0}{2}}\circ (I-\mathcal{L}_{G})^{\frac{w_0-w}{2}}(y_n)=\sum_{n}e_{n}'(y_n),
\end{align*}
which is the nuclear trace of $T_{a}:B^{w_0}_{p,q}\rightarrow B^{w_0}_{p,q} .$  Thus, we end the proof.
\end{proof} 
As an application of Theorem \ref{equivalencia}, we obtain the following Theorem on the $r$-nuclearity and the $r$-nuclear trace of Fourier multipliers on $B^{w}_{p,q}(G)$. In order to present our theorem, we recall the following  notation for the $l^r$-seminorm of matrices $A\in \mathbb{C}^{d\times d},$  $$\Vert A\Vert^{r}_{l^{r}}=\sum_{1\leq i,j\leq d}|a_{ij}|^{r} , 0<r\leq 1,$$ and we define the following function,
\begin{equation}
\varepsilon(t)=\frac{1}{2},\,\,1<t\leq 2, \text{  and  }\varepsilon(t)=\frac{1}{t}, \,\,t\geq 2,
\end{equation}
which arises of natural way in $L^{p}$-estimates of the entries $\xi_{ij}$ of  representations $[\xi]\in \widehat{G}.$ In fact, for $1\leq p\leq \infty,$ $\Vert \xi_{ij}\Vert_{L^{p}(G)}\leq d_{\xi}^{-\varepsilon(p)}.$ (See, Lemma 2.5 of \cite{DR}).
\begin{theorem}\label{TT1}
Let $G$ be a compact Lie group, $n=\dim(G),$ $0<r\leq 1$ and $T_{a}:C^{\infty}(G)\rightarrow C^{\infty}(G)$ be a operator with matrix valued symbol $a(\xi).$ Under almost one of the following conditions
\begin{itemize}
\item[(1).] $1<p<q<\infty$ and $$ \sum_{[\xi]\in \hat{G}}\langle \xi\rangle^{n(\frac{1}{p}-\frac{1}{q})r}\Vert a(\xi)\Vert^{r}_{l^{r}} d_{\xi}^{1+r(1-\varepsilon(p)-\varepsilon(q'))}<\infty.$$
\item[(2).] $1\leq p<\infty,$ $q=1$ and
   $$  \sum_{[\xi]\in \hat{G}} \langle \xi\rangle^{\frac{nr}{p}}\Vert a(\xi) \Vert_{l^r}^r d_{\xi}^{1+r(1-\varepsilon(p))}<\infty.  $$
\item[(3).]  $1<p=q\leq 2$ and $$ \sum_{[\xi]\in \hat{G}}\Vert a(\xi)\Vert_{l^r}^{r}d_{\xi}^{1+r(\frac{1}{p}-\frac{1}{2})}<\infty. $$

\item[(4).] $2=q\leq p<\infty$ and 
$$ \sum_{[\xi]\in \hat{G}}\Vert a(\xi)\Vert_{l^r}^{r}d_{\xi}^{1+r(\frac{1}{2}-\frac{1}{p})}<\infty. $$
\end{itemize}
The operator $T_{a}:B^{w}_{p,q}\rightarrow B^{w}_{p,q}$ extends to a $r$-nuclear operator for all $w\in\mathbb{R}.$ Moreover, the nuclear trace of $T_{a}$ is given by
\begin{equation}
\mathrm{Tr}(T_{a})=\sum_{[\xi]\in\widehat{G}}d_{\xi}\mathrm{Tr}[a(\xi)].
\end{equation}
\end{theorem}
\begin{proof}
Let $T_a$ be the Fourier multiplier given by 
\begin{equation}
T_af(x)=\sum_{[\xi]\in \widehat{G}}d_{\xi}\text{Tr}[\xi(x)a(\xi)(\mathscr{F}f)(\xi)].
\end{equation}
We observe that $$ \text{Tr}[\xi(x)a(\xi)(\mathscr{F}f)(\xi)]=\sum_{i,j.k=1}^{d_\xi}\xi(x)_{i,j}a(x,\xi)_{j,k}(\mathscr{F}f)(\xi)_{k,i}.  $$
Then, we can write
\begin{equation}
T_{a}f(x)=\sum_{[\xi],i,j,k}H_{\xi,i,j,k}(x)G_{\xi,i,j,k}(f)
\end{equation}
where
$$H_{\xi,i,j,k}(x)=d_{\xi}\xi(x)_{ij}a(\xi)_{i,k},\,\,\,G_{\xi,i,k}(f)=\widehat{f}(\xi)_{ki},\,\,\, 1\leq i,j,k\leq d_\xi.$$ 
Let us assume that $1<p,q<\infty.$ We will prove that $T_{a}:B^{w}_{p,q}(G)\rightarrow B^{w}_{p,q}(G)$ is a $r$-nuclear in every case above by showing that
\begin{equation}\label{nucleardeco1}
\sum_{[\xi]\in \hat{G}}\sum_{i,j,k=1}^{d_\xi} \Vert H_{\xi,i,j,k}\Vert^r_{B^{w}_{p,q}} \Vert G_{\xi,i,j,k} \Vert^r_{(B^{w}_{p,q})'}<\infty.
\end{equation}
Later,  considering the Theorem \ref{equivalencia} we deduce the nuclearity of $T_{a}$ on $B^{w}_{p,q}$ for every $w\in\mathbb{R}.$ First we estimate the Besov norm $\Vert H_{\xi,i,j,k}\Vert_{B^{w}_{p,q}}$ as follows:
\begin{align*}
\Vert H_{\xi,i,j,k}\Vert^{q}_{B^{w}_{p,1}}&=\sum_{s=0}^{\infty}2^{swq}\left\Vert \sum_{2^{s}\leq\langle\eta\rangle <2^{s+1} } d_{\eta}\mathrm{Tr}[\eta(x)\widehat{H}_{\xi,i,j,k}(\eta)]
\right\Vert^q_{L^p}\\
&=\sum_{s=0}^{\infty}2^{swq}\left\Vert \sum_{2^{s}\leq \langle\eta\rangle <2^{s+1} } d_{\eta}\mathrm{Tr}[\eta(x)\int_{G}\eta(y)^{*}d_{\xi}\xi(y)_{ij}a(\xi)_{jk}dy]
\right\Vert^q_{L^p}\\
&=\sum_{s=0}^{\infty}2^{swq}\left\Vert \sum_{2^{s}\leq\langle\eta\rangle <2^{s+1} }\sum_{u,v=1}^{d_\eta} d_{\eta}[\eta(x)_{uv}\int_{G}\eta(y)^{*}_{vu}d_{\xi}\xi(y)_{ij}a(\xi)_{jk}dy]
\right\Vert^q_{L^p}\\
&=\sum_{s=0}^{\infty}2^{swq}\left\Vert \sum_{2^{s}\leq\langle\eta\rangle <2^{s+1} }\sum_{u,v=1}^{d_\eta} d_{\eta}[\eta(x)_{uv}\int_{G}\overline{\eta(y)}_{uv}d_{\xi}\xi(y)_{ij}a(\xi)_{jk}dy]
\right\Vert^q_{L^p}\\
&=\sum_{s=0}^{\infty}2^{swq}\left\Vert \sum_{2^{s}\leq\langle\eta\rangle <2^{s+1} }\sum_{u,v=1}^{d_\eta} d_{\eta}[\eta(x)_{uv}d_{\xi}a(\xi)_{jk}\langle \xi_{ij},\eta_{uv}\rangle_{L^2(G)}]
\right\Vert^{q}_{L^p}\\
&=\sum_{s=0}^{\infty}2^{swq}\left\Vert \sum_{2^{s}\leq\langle\eta\rangle <2^{s+1} }\sum_{u,v=1}^{d_\eta} d_{\eta}[\eta(x)_{uv}d_{\xi}a(\xi)_{jk}d^{-1/2}_{\xi}d^{-1/2}_{\eta}\delta_{(\xi,i,j),(\eta,u,v)}]
\right\Vert^q_{L^p}.\\
\end{align*}
Let us chose the most smallest $s_{\xi}\in\mathbb{N}$ such that $2^{s_{\xi}}\leq\langle \xi\rangle<2^{s_{\xi}+1},$ considering that $\delta_{(\xi,i,j),(\eta,u,v)}=1$ only if $\xi=\eta, u=i$ and $v=j$ and $\delta_{(\xi,i,j),(\eta,u,v)}=0$ in other case, we obtain
$$ \Vert H_{\xi,i,j,k}\Vert_{B^{w}_{p,q}}= 2^{s_{\xi}w}d_{\xi}\Vert\xi_{ij} \Vert_{L^p}|a(\xi)_{jk}|. $$
Since the $L^p$ norm of $\xi_{ij}$ can be estimate by $d_{\xi}^{-\varepsilon(p)}$ for $1<p<\infty$ and by considering that $2^{s_{\xi}}\leq \langle \xi \rangle<2^{s_\xi+1}$ we get
$$ \Vert H_{\xi,i,j,k}\Vert_{B^{w}_{p,q}}\leq \langle \xi\rangle^{w}|a(\xi)_{jk}|d_{\xi}^{1-\varepsilon(p)}. $$
If $1<p<q<\infty$ and $w=n(\frac{1}{p}-\frac{1}{q}),$ we have  the embedding $B^{w}_{p,q}\hookrightarrow L^{q}$ and  we get,
\begin{align*}
\Vert G_{\xi,i,j,k} \Vert_{(B^{w}_{p,q})^{'}}&=\sup_{\Vert f\Vert_{B^{w}_{p,q}}=1}|\widehat{f}(\xi)_{ki}|\leq \sup_{\Vert f\Vert_{B^{w}_{p,q}}=1}|\int_{G}\xi(x)^{*}_{ki}f(x)dx|\\
&\leq \sup_{\Vert f\Vert_{B^{w}_{p,q}}=1} \Vert \xi^{*}_{ki} \Vert_{L^{q'}}\Vert f\Vert_{L^q}\lesssim \sup_{\Vert f\Vert_{B^{w}_{p,q}}=1}d_{\xi}^{-\varepsilon(q')}\Vert f\Vert_{L^q}\\
&\leq \sup_{\Vert f\Vert_{B^{w}_{p,q}}=1} d_{\xi}^{-\varepsilon(q')}\Vert f \Vert_{B^w_{p,q}}\\
&=d_{\xi}^{-\varepsilon(q')}.
\end{align*}
Now we estimate \eqref{nucleardeco1} as follows.
\begin{align*}
\sum_{[\xi]\in \hat{G}}\sum_{i,j,k=1}^{d_\xi} \Vert H_{\xi,i,j,k}\Vert^{r}_{B^{w}_{p,q}} \Vert G_{\xi,i,j,k} \Vert^{r}_{(B^{w}_{p,q})'} &\lesssim \sum_{[\xi]\in \hat{G}}\sum_{i,j,k=1}^{d_\xi} \langle \xi\rangle^{wr}|a(\xi)_{jk}|^{r}d_{\xi}^{r(1-\varepsilon(p)-\varepsilon(q'))}\\
&\lesssim \sum_{[\xi]\in \hat{G}}\sum_{j,k=1}^{d_\xi} \langle \xi\rangle^{wr}|a(\xi)_{jk}|^{r}d_{\xi}^{1+r(1-\varepsilon(p)-\varepsilon(q'))}\\
&\leq \sum_{[\xi]\in \hat{G}}\langle \xi\rangle^{wr}\Vert a(\xi)\Vert^{r}_{l^{r}} d_{\xi}^{1+r(1-\varepsilon(p)-\varepsilon(q'))}<\infty.
\end{align*}
If we consider $1\leq p<\infty$ and $q=1$ we have the embedding $B^{w}_{p,1}\hookrightarrow L^{\infty}$ for $w=\frac{n}{p},$ and taking into account that $\Vert \xi_{ki}\Vert_{L^{\infty}}\leq 1$ we deduce the estimates
$$ \Vert H_{\xi,i,j,k}\Vert_{B^{w}_{p,1}}\leq \langle \xi\rangle^{w}|a(\xi)_{jk}|d_{\xi}^{1-\varepsilon(p)},\hspace{0.5cm}\Vert G_{\xi,i,j,k} \Vert_{(B^{w}_{p,1})^{'}}\leq \sup_{B^{\frac{n}{p}}_{p,\infty}=1}\Vert \xi_{ki}^{*}\Vert_{L^\infty} \Vert f \Vert_{L^{\infty}}\lesssim 1.  $$
So, we have
\begin{align*}
\sum_{[\xi]\in \hat{G}}\sum_{i,j,k=1}^{d_\xi} \Vert H_{\xi,i,j,k}\Vert^{r}_{B^{w}_{p,1}} \Vert G_{\xi,i,j,k} \Vert^{r}_{(B^{w}_{p,1})'} &\lesssim \sum_{[\xi]\in \hat{G}} \langle \xi\rangle^{wr}\Vert a(\xi) \Vert_{l^r}^r d_{\xi}^{1+r(1-\varepsilon(p))}<\infty.
\end{align*}
The case where $1<p\leq 2$ we have the embedding $B^{w}_{p,p}\hookrightarrow H^{w,p},$ for every $w\in \mathbb{R}.$ In particular, with $w=0$ we have the estimates
\begin{align*}
\sum_{[\xi]\in \hat{G}}\sum_{i,j,k=1}^{d_\xi} \Vert H_{\xi,i,j,k}\Vert^{r}_{B^{0}_{p,p}} \Vert G_{\xi,i,j,k} \Vert^{r}_{(B^{0}_{p,p})'} &\lesssim \sum_{[\xi]\in \hat{G}}\Vert a(\xi)\Vert_{l^r}^{r}d_{\xi}^{1+r(1-\frac{1}{2}-\frac{1}{p'})}\\
&\lesssim \sum_{[\xi]\in \hat{G}}\Vert a(\xi)\Vert_{l^r}^{r}d_{\xi}^{1+r(\frac{1}{p}-\frac{1}{2})}<\infty.
\end{align*}
 Now, for $q=2\leq p<\infty$ we use the embedding $B^{0}_{p,2}\hookrightarrow H^{0,p}$ in order to obtain
 \begin{align*}
\sum_{[\xi]\in \hat{G}}\sum_{i,j,k=1}^{d_\xi} \Vert H_{\xi,i,j,k}\Vert^{r}_{B^{0}_{p,2}} \Vert G_{\xi,i,j,k} \Vert^{r}_{(B^{0}_{p,2})'} &\lesssim \sum_{[\xi]\in \hat{G}}\Vert a(\xi)\Vert_{l^r}^{r}d_{\xi}^{1+r(1-\frac{1}{p}-\frac{1}{2})}\\
&\lesssim \sum_{[\xi]\in \hat{G}}\Vert a(\xi)\Vert_{l^r}^{r}d_{\xi}^{1+r(\frac{1}{2}-\frac{1}{p})}<\infty.
\end{align*}

So, in every specific case, we have proved that $T_a$ is nuclear on $B^{w}_{p,q}$ and therefore on every $B^{w}_{p,1}$ with $w\in\mathbb{R}.$ Now, we compute the nuclear trace of $T_a.$  Since $B^{w}_{p,q}$ has the approximation property, we deduce that the nuclear trace of $T_{a}$ is well defined, this means that it can be computed from any nuclear decomposition.  So we get
\begin{align*}
\mathrm{Tr}(T_{a}) &=\sum_{[\xi],i,j,k} G_{\xi,i,k}(H_{\xi,i,j,k})= \sum_{[\xi],i,j,k}\mathscr{F}(H_{\xi,i,j,k})(\xi)_{k,i}.
\end{align*}
By using the definition of Fourier transform, we obtain
$$  \mathscr{F}(H_{\xi,i,j,k})(\xi)=\int_{G}d_{\xi}\xi(x)^*\xi(x)_{i,j}a(\xi)_{j,k}dx.$$
Hence
$$  \mathscr{F}(H_{\xi,i,j,k})(\xi)_{ki}=\int_{G}d_{\xi}\xi(x)^*_{k,i}\xi(x)_{i,j}a(\xi)_{j,k}dx.$$
Using this fact, we deduce that
\begin{align*}
\mathrm{Tr}(T_{a})&=\sum_{[\xi],i,j,k}d_{\xi}\int_{G}\xi(x)^*_{k,i}\xi(x)_{i,j}a(\xi)_{j,k}dx\\
&=\sum_{[\xi]\in\widehat{G}}d_{\xi}\int_{G}\sum^{d_\xi}_{i,j,k=1} \xi(x)^*_{k,i}\xi(x)_{i,j}a(\xi)_{j,k}dx\\
&=\sum_{[\xi]\in \widehat{G}}d_{\xi}\int_{G}\mathrm{Tr}[\xi(x)a(\xi)\xi(x^{*})]\\
&=\sum_{[\xi]\in \widehat{G}}d_{\xi}\mathrm{Tr}[a(\xi)].
\end{align*}
Thus, we end the proof.
\end{proof}
\begin{remark}
Now, we discuss the theorem above in relation with results in $L^{2}$ spaces. We observe that the result obtained when $p=q=2$ in the condition $(3)$  of Theorem \ref{TT1}, is most weak that Theorem 3.1 of  \cite{DR} where the condition $$ \sum_{[\xi]\in \hat{G}}\Vert a(\xi)\Vert_{S_{r}}^{r}d_{\xi}<\infty, $$
is imposed in terms of the $r$-Schatten seminorm $\Vert a(\xi)\Vert_{S_{r}}$ in order to obtain $r$-nuclearity. It was mentioned in \ref{TT1} that such condition is also necessary for the $r$-nuclearity of $T_{a}$. There exists two cases where both results are equivalent. One is, the case where the operator $T_a:C^{\infty}(G)\rightarrow C^{\infty}(G)$ is formally self-adjoint. In fact, with such condition in mind, one can to assume that the corresponding symbol $a(\xi)$ is diagonal by choosing a suitable basis in the representations spaces. In a such case, $\Vert a(\xi) \Vert_{l^{r}}=\Vert(\xi)\Vert_{S_r}.$ The other case arises when $G=\mathbb{T}^{n}$ is some $n$-dimensional torus, where $\widehat{G}=\mathbb{Z}^{n}$ and for every $\xi\in \mathbb{Z}^{n},$ $\Vert a(\xi) \Vert_{l^{r}}=\Vert(\xi)\Vert_{S_r}=|a(\xi)|.$ It is important to mention that the trace formula obtained above coincides with ones for $r$-Fourier multiplier in $L^p$ spaces obtained in \cite{DR,DR-1}. 
\end{remark}

We end this section with the following two examples on the nuclearity of suitable powers of the Bessel's potential and the heat operator.
\begin{example}
Let $G$ be a compact Lie group, $\mathcal{L}_{G}$ be the Laplace-Beltrami operator on $G$ and $n=\dim(G).$ We note that as consequence of Theorem \ref{TT1}, if $\alpha> n$ and $1<p\leq 2,$ the operator $T_{a}=(1-\mathcal{L}_{G})^{-\frac{\alpha}{2}}$ is nuclear on $B^{w}_{p,1}(G)$ for all $-\infty<w<\infty.$ Indeed, this operator has symbol $a(\xi)$  satisfying the condition $(3)$ of Theorem \ref{TT1} and as consequence we get 
$$ \sum_{[\xi]\in \hat{G}}d_{\xi} \Vert a(\xi)\Vert_{l^1}:=\sum_{[\xi]\in \hat{G}}d_{\xi}^{2} \langle \xi\rangle^{-\alpha} <\infty.$$
In this case $$\mathrm{Tr}[(1-\mathcal{L}_{G})^{-\frac{\alpha}{2}}]=\sum_{[\xi]\in\widehat{G}}d_\xi^{2}\langle \xi\rangle^{-\alpha}. $$
\end{example}
\begin{example}
For $t>0$ the heat operator is defined by $e^{-t\mathcal{L}_{G}}$ and its symbol is given by $a_{t}(\xi)=e^{-t\lambda_{[\xi]}}I_{d_\xi}.$ Clearly this symbol satisfies the hypotheses of Theorem \ref{TT1}, and $T_{a}=e^{-t\mathcal{L}_{G}}$ is a nuclear operator on $B^{w}_{p,1}(G)$ for all $-\infty<w<\infty.$ For the heat operator, the nuclear trace is $$\mathrm{Tr}[e^{-t\mathcal{L}_{G}}]=\sum_{[\xi]\in\widehat{G}}d_\xi^{2}e^{-t\lambda_{[\xi]}}. $$
\end{example}
\section{$r$-nuclear pseudo-differential operator on periodic Besov spaces}

In this section we present our results on the $r$-nuclearity of pseudo-differential operators on periodic Besov spaces. We use the notation of periodic pseudo-differential operators as developed in \cite{Ruz}.  Let us denote by  $\mathcal{S}(\mathbb{Z}^n)$ the Schwartz space of functions $\phi:\mathbb{Z}^n\rightarrow \mathbb{C}$ such that 
 \begin{equation}
 \forall M\in\mathbb{R}, \exists C_{M}>0,\, |\phi(\xi)|\leq C_{M}\langle \xi \rangle^M,
 \end{equation}
where $\langle \xi \rangle=(1+|\xi|^2)^{\frac{1}{2}}$. The toroidal Fourier transform is defined, for any $f\in C^{\infty}(\mathbb{T}^n)$, by $\widehat{f}(\xi)=\int_{}e^{-i2\pi\langle x,\xi\rangle}f(x)dx$, where $\xi\in\mathbb{Z}^n$, and the inversion formula is given by $f(x)=\sum_{}e^{i2\pi\langle x,\xi \rangle }\widehat{u}(\xi)$, for $ x\in\mathbb{T}^n$. The periodic H\"ormander class $S^m_{\rho,\delta}(\mathbb{T}^n\times \mathbb{R}^n), \,\, 0\leq \rho,\delta\leq 1,$ consists of those functions $a(x,\xi)$ which are smooth in $(x,\xi)\in \mathbb{T}^n\times \mathbb{R}^n$ and which satisfy toroidal symbols inequalities
\begin{equation}
|\partial^{\beta}_{x}\partial^{\alpha}_{\xi}a(x,\xi)|\leq C_{\alpha,\beta}\langle \xi \rangle^{m-\rho|\alpha|+\delta|\beta|}.
\end{equation}
Symbols in $S^m_{\rho,\delta}(\mathbb{T}^n\times \mathbb{R}^n)$ are symbols in $S^m_{\rho,\delta}(\mathbb{R}^n\times \mathbb{R}^n)$ (see \cite{Ruz}) of order $m$ which are $1$-periodic in $x.$
If $a(x,\xi)\in S^{m}_{\rho,\delta}(\mathbb{T}^n\times \mathbb{R}^n),$ the corresponding pseudo-differential operator is defined by
\begin{equation}
T_au(x)=\int_{\mathbb{T}^n}\int_{\mathbb{R}^n}e^{i2\pi\langle x-y,\xi \rangle}a(x,\xi)u(y)d\xi dy.
\end{equation}
The set $S^m_{\rho,\delta,\,\nu,\mu}(\mathbb{T}^n\times \mathbb{Z}^n),\, 0\leq \rho,\delta\leq 1,$ $\nu, \mu\in \mathbb{N}$, consists of  those functions $a(x, \xi)$ which are smooth in $x$  for all $\xi\in\mathbb{Z}^n$ and which satisfy

\begin{equation}
 |\Delta^{\alpha}_{\xi}\partial^{\beta}_{x}a(x,\xi)|\leq C_{\alpha,\beta}\langle \xi \rangle^{m-\rho|\alpha|+\delta|\beta|},\,\,\ |\alpha|\leq \nu,\,|\beta|\leq \mu.
 \end{equation}
The operator $\Delta$ is the difference operator defined in \cite{Ruz}. The toroidal operator with symbol $a(x,\xi)$ is defined as
\begin{equation}
a(x,D_{x})u(x)=\sum_{\xi\in\mathbb{Z}^n}e^{i 2\pi\langle x,\xi\rangle}a(x,\xi)\widehat{u}(\xi),\,\, u\in C^{\infty}(\mathbb{T}^n).
\end{equation}
Besov spaces have been introduced in Section 2 for general compact Lie groups. Now we present this notion for the toroidal case;  let $w\in\mathbb{R},$ $0< q<\infty$ and $0<p\leq \infty.$ If $f$ is a measurable function on $\mathbb{T}^n,$ we say that $f\in B^w_{p,q}(\mathbb{T}^n)$ if $f$ satisfies
\begin{equation}\Vert f \Vert_{B^w_{p,q}}:=\left( \sum_{m=0}^{\infty} 2^{mwq}\Vert \sum_{2^m\leq |\xi|< 2^{m+1}}  e^{i2\pi x\cdot \xi}\widehat{f}(\xi)\Vert^q_{L^p(\mathbb{T}^n)}\right)^{\frac{1}{q}}<\infty.
\end{equation}
If $q=\infty,$ $B^w_{p,\infty}(\mathbb{T}^n)$ consists of those functions $f$ satisfying
\begin{equation}\Vert f \Vert_{B^w_{p,\infty}}:=\sup_{m\in\mathbb{N}} 2^{mw}\Vert \sum_{2^m\leq |\xi|< 2^{m+1}}  e^{i2\pi x\xi}\widehat{f}(\xi)\Vert_{L^p(\mathbb{T}^n)}<\infty.
\end{equation}
In the case of $p=q=\infty$ and $0<w<1=n $ we obtain $B^w_{\infty,\infty}(\mathbb{T})=\Lambda^w(\mathbb{T}),$ that is the H\"older space of order $\omega;$ these are Banach spaces together with the norm $$\Vert f\Vert_{\Lambda^{w}}=\sup_{x,h\in\mathbb{T}}|f(x+h)-f(x)||h|^{-w}+\sup_{x\in \mathbb{T}}|f(x)|.$$
For $1\leq p\leq \infty$ and $1\leq q\leq \infty,$  $B^w_{p,q}(\mathbb{T}^n)$ are Banach spaces. Moreover, if $w\in\mathbb{R}$ we have the  identity of Hilbert spaces $H^{w}_{2,2}(\mathbb{T}^n)=B^{r}_{2,2}(\mathbb{T}^n)$ of Besov spaces with Sobolev spaces. When studying the orders of periodic pseudo-differential operators, estimates
for the  Fourier coefficients of symbols are useful,
and therefore we present an auxiliary result on this subject.
\begin{lemma}\label{lem} Let $0\leq \rho,\delta \leq 1.$  Assume that $a\in S^{m}_{\rho,\delta,u,2k}(\mathbb{T}^n\times \mathbb{Z}^n).$ Let $\widehat{a}(\eta,\cdot)$ be the
Fourier transform of the symbol with respect to x, i.e., the Fourier transform of the smooth function $x\mapsto a(x,\cdot)$. Then, for every $k\in \mathbb{N}$ we have
\begin{equation}
|\widehat{a}(\eta,\xi)|\leq C\langle  \eta \rangle^{-2k}\langle \xi\rangle^{m+\delta|2k|}.
\end{equation}
\end{lemma}
\begin{proof}
The proof can be found in \cite{Ruz}, Lemma 4.2.1.
\end{proof}

With notation above, we present our results on the $r-$nuclearity of periodic pseudo-differential operators. We reserve the notation $A\lesssim B$ if there exists $c>0$ independent of $A$ and $B$ such that $A\leq c\cdot B.$ The conjugate exponent $p'$ of $p,$ $1\leq p\leq \infty$ is defined by $1/p+1/p'=1.$ Our starting point is the following result (Theorem 6.2 of \cite{RuzBesov}).
\begin{lemma}\label{cardona}
Let $1<p_1\leq 2,$ $\alpha>0,$ $n\in\mathbb{N}$ and $\beta=(\alpha +\frac{1}{p'_1})^{-1}.$ If $w_1=\alpha n$ then the Fourier transform is a bounded operator from $B_{p_1,\beta}^{\alpha n}(\mathbb{T}^n)$ into $L^{\beta}(\mathbb{Z}^n).$
\end{lemma} 
By using the lemma above, we obtain the following result on the $r$-nuclearity of pseudo-differential operators on periodic Besov spaces. 
 \begin{theorem}\label{t1}
Let $0<r\leq 1,$  $0<\alpha\leq \frac{1}{2},$ $0\leq \rho,\delta \leq 1,$ and  $k>\frac{n}{2},$ $k\in\mathbb{N}.$  Let us consider $a\in S^{m}_{\rho,\delta, 0,2k}(\mathbb{T}^n\times \mathbb{Z}^n).$ Under the following conditions,
\begin{itemize}
\item $w_{1}=\alpha \cdot n, \,\,1< p_1\leq 2,$ $q_1=(\alpha+\frac{1}{p_1'})^{-1}.$
\item $0\leq w_{2}<2k-n,\,\, m<-\frac{n}{r}-w_2-\delta(2k),$  $1\leq p_2\leq \infty,$ $1\leq q_2\leq  \infty,$
\end{itemize}
the pseudo-differential operator $T_a: B^{w_1}_{p_1,q_1}(\mathbb{T}^n)\rightarrow B^{w_2}_{p_2,q_2}(\mathbb{T}^n) $ is $r$-nuclear.
\end{theorem} 
\begin{proof}
We begin by writing 
\begin{equation}
T_af(x)=\sum_{\xi\in\mathbb{Z}^n}e^{i2\pi x\cdot \xi}a(x,\xi)\widehat{f}(\xi)=\sum_{\xi\in\mathbb{Z}^n}G_{\xi}(f)H_{\xi}(x),
\end{equation}
where $H_{\xi}(x)=e^{i2\pi x\xi}a(x,\xi)$ and $G_{\xi}(f)=\widehat{f}(\xi).$ By Lemma \ref{cardona}, for every $\xi\in \mathbb{Z}^n$ we have
\begin{align*}
\Vert G_{\xi} \Vert_{(B^{w_1}_{p_1,q_1})'}&=\sup\{|G_{\xi}(f)|: \Vert f \Vert_{B^{w_1}_{p_1,q_1}} =1\}\\
&=\sup\{|\widehat{f}(\xi)|: \Vert f \Vert_{B^{w_1}_{p_1,q_1}} =1\}\\
&\leq \sup\{\Vert \widehat{f}\Vert_{L^{\beta}(\mathbb{Z}^n)}: \Vert f \Vert_{B^{w_1}_{p_1,q_1}} =1\}\\
&\lesssim 1.
\end{align*}
Now, if $1\leq q_2<\infty$ we observe that
\begin{align*}
\Vert H_{\xi} \Vert^{q_2}_{B^{w_2}_{p_2, q_2}}=\sum_{s=0}^{\infty} 2^{sw_2 q_2}\Vert     \sum_{2^s\leq |\eta|< 2^{s+1}}   e^{i2\pi\eta y}\widehat{H_{\xi}}(\eta)  \Vert^{q_2}_{L^{p_2}},
\end{align*}
where
\begin{align*} \widehat{H_{\xi}}(\eta) &=\int_{\mathbb{T}^n}e^{-i2\pi x\eta}e^{i2\pi x\xi}a(x,\xi)dx\\
&=\widehat{a}(\eta-\xi,\xi).
\end{align*}
Hence, we have
\begin{align*}
\Vert H_{\xi} \Vert_{B^{w_2}_{p_2, q_2}}
 &\leq \left(\sum_{s=0}^{\infty} 2^{sw_2 q_2} \left[      \sum_{2^s\leq |\eta|< 2^{s+1}}   |\widehat{a}(\eta-\xi,\xi)|\right]^{q_2}\right)^{1/q_2}\\
 &\leq \left(\sum_{s=0}^{\infty}  \left[      \sum_{2^s\leq |\eta|< 2^{s+1}}  |\eta|^{w_2} |\widehat{a}(\eta-\xi,\xi)|\right]^{q_2} \right)^{1/q_2}\\
 &=\Vert  F\Vert_{L^{q_2}(\mathbb{N})}
\end{align*}
where $F$ is the sequence on $\mathbb{N}$ given by
$$ F(s)=\sum_{2^s\leq |\eta|< 2^{s+1}}  |\eta|^{w_2} |\widehat{a}(\eta-\xi,\xi)|. $$
Since $L^{1}(\mathbb{N})\subset L^{q_2}(\mathbb{N})$ is a continuous inclusion, we have $\Vert F \Vert_{L^{q_2}}\lesssim \Vert F \Vert_{L^{1}}. $ Hence, 
 \begin{align*}
 \Vert H_{\xi} \Vert_{B^{w_2}_{p_2, q_2}}
 &\lesssim \sum^{\infty}_{s=0}\sum_{2^s\leq |\eta|< 2^{s+1}}  |\eta|^{w_2} |\widehat{a}(\eta-\xi,\xi)|\\
 &\leq \sum_{\eta\in\mathbb{Z}^n}  |\eta|^{w_2} |\widehat{a}(\eta-\xi,\xi)|.
\end{align*}
If $q_2=\infty,$ by definition of Besov norm we have
\begin{align*}
\Vert H_{\xi} \Vert_{B^{w_2}_{p_2, q_2}} &=\sup_{s\in \mathbb{N}} 2^{sw_2}\Vert     \sum_{2^s\leq |\eta|< 2^{s+1}}   e^{i2\pi\eta y}\widehat{H_{\xi}}(\eta)  \Vert_{L^{p_2}}\\
&=\sup_{s\in \mathbb{N}} 2^{sw_2}\Vert     \sum_{2^s\leq |\eta|< 2^{s+1}}   e^{i2\pi\eta y}\widehat{a}(\eta-\xi,\xi)  \Vert_{L^{p_2}}\\
&\leq \sup_{s\in \mathbb{N}} 2^{sw_2}     \sum_{2^s\leq |\eta|< 2^{s+1}}  |\widehat{a}(\eta-\xi,\xi)|\\ 
&\leq \sup_{s\in \mathbb{N}}      \sum_{2^s\leq |\eta|< 2^{s+1}} |\eta|^{w_2} |\widehat{a}(\eta-\xi,\xi)|\\
&\leq   \sum_{ \eta\in\mathbb{Z}^n} |\eta|^{w_2} |\widehat{a}(\eta-\xi,\xi)|.  
\end{align*}
Now, by Lemma \ref{lem} we have  $|\widehat{a}(\eta-\xi,\xi)|\leq C\langle\eta-\xi \rangle^{-2k}\langle \xi \rangle^{m+\delta(2k)}. $ On the other hand, by Peetre's inequality (Proposition 3.3.31 of \cite{Ruz}) we can write $$\vert \eta \vert^{w_2} \lesssim \langle \eta \rangle^{w_2}\lesssim \langle \eta-\xi \rangle^{w_2}\langle \xi\rangle^{w_2} $$
From this,  for all $1\leq q_2\leq \infty$  we obtain,
\begin{align*}
\Vert H_{\xi} \Vert_{B^{w_2}_{p_2, q_2}}
 &\leq \sum_{\eta\in\mathbb{Z}^n} \langle \eta-\xi \rangle^{w_2-2k}\langle \xi\rangle^{w_2+m+\delta(2k)}\\
 &=\langle \xi\rangle^{w_2+m+\delta(2k)} \sum_{\eta\in\mathbb{Z}^n} \langle \eta \rangle^{w_2-2k}.
\end{align*}
From the condition $0\leq w_2<2k-n $ we deduce the convergence of the series $$ \sum_{\eta\in\mathbb{Z}^n} \langle \eta \rangle^{w_2-2k}. $$
Hence, we have
\begin{align}
\sum_{\xi\in\mathbb{Z}^n} \Vert H_{\xi} \Vert^{r}_{B^{w_2}_{p_2, q_2}}\Vert G_{\xi} \Vert^{r}_{  ( B^{w_1}_{p_1, q_1} )'}\lesssim \sum_{\xi\in\mathbb{Z}^n} \langle \xi\rangle^{r(w_2+m+\delta(2k))}<\infty.  
\end{align}
Since $0<\alpha \leq \frac{1}{2},$ we deduce that $q_{1}\geq 1.$ Hence $B^{w_1}_{p_1,q_1}$ is a Banach space. So, $T_a: B^{w_1}_{p_1,q_1}(\mathbb{T}^n)\rightarrow B^{w_2}_{p_2,q_2}(\mathbb{T}^n) $ is a $r$-nuclear operator.
\end{proof}
In the previous theorem the $r$-nuclearity has been established for $w_{2}\geq 0.$ In the following theorem we provide an analysis of the problem for $w_{2}<0.$ 
\begin{theorem}\label{t2222}
Let $0\leq \delta,\rho\leq 1,$  and $0<\alpha\leq \frac{1}{2}.$ If $a\in S^{m}_{\rho,\delta,0,2k}$,  under the following conditions
\begin{itemize}
\item $w_{1}=\alpha \cdot n, \,\,1< p_1\leq 2,$ $q_1=(\alpha+\frac{1}{p_1'})^{-1},$
\item $-\infty<w_2<-\frac{n}{2},$ $1\leq p_2,q_2\leq \infty,$ $m\leq -\delta(2k),$  $k>n/4,$
\end{itemize}
the operator $T_{a}: B^{w_1}_{p_1,q_1}\rightarrow B^{w_2}_{p_2,q_2} $ is nuclear. Moreover, if we assume
\begin{itemize}
\item $w_{1}=\alpha \cdot n, \,\,1< p_1\leq 2,$ $q_1=(\alpha+\frac{1}{p_1'})^{-1},$
\item $w_2\leq 0,$ $1\leq p_2,q_2\leq \infty,$ $m<-\frac{n}{r}- \delta(2k),$  $k>n/2,$ 
\end{itemize} the operator $T_a$ is $r$-nuclear for all $0<r\leq 1.$
\end{theorem}
\begin{proof} From the proof of Theorem \ref{t1}, we have that $\Vert G_{\xi} \Vert_{B^{w_1}_{p_1,q_1}}\lesssim 1$ and
\begin{align*}
 \Vert H_{\xi} \Vert_{B^{w_2}_{p_2, q_2}}
 &\lesssim \sum^{\infty}_{s=0}\sum_{2^s\leq |\eta|< 2^{s+1}}  |\eta|^{w_2} |\widehat{a}(\eta-\xi,\xi)|\\
 &\leq \sum_{\eta\neq 0}  |\eta|^{w_2} |\widehat{a}(\eta-\xi,\xi)|.
\end{align*}
If $-\infty<w_2<-n/2,$ we deduce that
\begin{align*}
 \Vert H_{\xi} \Vert_{B^{w_2}_{p_2, q_2}}
 &\lesssim  \sum_{\eta\in\mathbb{Z}^n}  \langle \eta\rangle^{w_2} |\widehat{a}(\eta-\xi,\xi)|\\
 &\lesssim  \sum_{\eta\in\mathbb{Z}^n}  \langle \eta\rangle^{w_2}\langle \eta-\xi \rangle^{-2k}\langle \xi\rangle^{m+\delta(2k)}\\
&= \langle \xi\rangle^{m+\delta(2k)} [\langle \,\cdot\, \rangle^{w_{2}} \ast \langle \,\cdot\,\rangle^{-2k} ](\xi).
 \end{align*}
 By the Young's inequality, $\langle \eta\rangle^{w_2}\in L^{2}(\mathbb{Z})$ and $\langle \eta\rangle^{-2k}\in L^{2}(\mathbb{Z})$ implies that $$ \langle \,\cdot\, \rangle^{w_{2}} \ast \langle \,\cdot\,\rangle^{-2k}\in L^{1}(\mathbb{Z}).$$ With this in mind, using the fact that $\Vert G_{\xi} \Vert_{B^{w_1}_{p_1,q_1}}\lesssim 1,$ and the condition $m\leq -\delta(2k)$ we obtain
 \begin{align}
\sum_{\xi\in\mathbb{Z}^n} \Vert H_{\xi} \Vert_{B^{w_2}_{p_2, q_2}}\Vert G_{\xi} \Vert_{  ( B^{w_1}_{p_1, q_1} )'}\lesssim \sum_{\xi\in\mathbb{Z}^n} [\langle \,\cdot\, \rangle^{w_{2}} \ast \langle \,\cdot\,\rangle^{-2k} ](\xi)<\infty.  
\end{align}
This inequality implies the nuclearity of $T_{a}.$ Now, we will treat the case $w_2\leq 0,$ $k>n/2.$ In fact, we have
\begin{align*}
 \Vert H_{\xi} \Vert_{B^{w_2}_{p_2, q_2}}
 &\lesssim  \sum_{\eta\in\mathbb{Z}^n}  \langle \eta\rangle^{w_2} |\widehat{a}(\eta-\xi,\xi)|\\
 &\lesssim  \sum_{\eta\in\mathbb{Z}^n} \langle \eta-\xi \rangle^{-2k}\langle \xi\rangle^{m+\delta(2k)}\\
&\lesssim \langle \xi\rangle^{m+\delta(2k)}
 \end{align*}
 Thus,
 \begin{align}
\sum_{\xi\in\mathbb{Z}^n} \Vert H_{\xi} \Vert^r_{B^{w_2}_{p_2, q_2}}\Vert G_{\xi} \Vert^r_{  ( B^{w_1}_{p_1, q_1} )'}\lesssim \sum_{\xi\in\mathbb{Z}^n} \langle \xi\rangle^{r(m+\delta\cdot 2k)} <\infty. 
\end{align}This proves the $r-$nuclearity of $T_a$ when $m<-\frac{n}{r}-\delta\cdot 2k.$
\end{proof}

In order to get $r$-nuclearity of operators from H\"older into Besov spaces, we recall the following lemma (see \cite{blo, Oto1, Oto2}.)

\begin{lemma}\label{general}
 Let $1\leq p\leq 2$ and let $s_p=1/p-1/2.$ Then, the Fourier transform $f\mapsto \mathscr{F}f$ from $\Lambda^s(\mathbb{T})$ into $L^{p}(\mathbb{T})$ is a bounded operator for all $s,$ $s_p<s<1.$ In particular, if $p=1$ we obtain the Bernstein Theorem.
\end{lemma}
Now we study the $r$-nuclearity of periodic operators from H\"older spaces (resp. Lebesgue) into Besov spaces.
\begin{theorem}\label{holder}
Let $0<r\leq 1,$  $0\leq \rho,\delta \leq 1,$ and  $k>\frac{n}{2},$ $k\in\mathbb{N}.$  Let us consider $a\in S^{m}_{\rho,\delta, 0,2k}(\mathbb{T}^n\times \mathbb{Z}^n).$ Under the following conditions,
\begin{itemize}
\item $X^{(n)}=L^{p}(\mathbb{T}^n),$ $1\leq p\leq 2$ or $X^{(1)}=B^{s}_{\infty,\infty}(\mathbb{T}^1),$ $0<s<1.$
\item $0\leq w_{2}<2k-n,\,\, m<-\frac{n}{r}-w_2-\delta(2k),$  $1\leq p_2\leq \infty,$ $1\leq q_2\leq  \infty,$
\end{itemize}
the pseudo-differential operator $T_a: X^{(n)} \rightarrow B^{w_2}_{p_2,q_2}(\mathbb{T}^n) $ is $r$-nuclear.
\end{theorem}
\begin{proof} If we assume that $X:=X(\mathbb{T}^n)$ is a Banach space of periodic functions with the property that $\Vert \widehat{f}\Vert_{L^{\infty}(\mathbb{Z}^n)} \leq C\Vert f\Vert_{X(\mathbb{T}^n)},$ then we obtain
$$ \Vert G_{\xi}\Vert_{X'}:=\sup_{\Vert f\Vert_{X}=1}|\widehat{f}(\xi)|\leq C. $$
Now, we note that in the following cases, $X$ has the mentioned property:
\begin{itemize}
\item $X=L^{1}(\mathbb{T}^n).$ (As a consequence of $\Vert \widehat{f}\Vert_{L^\infty}\leq \Vert f \Vert_{L^1}.$)
\item $X=L^{p}(\mathbb{T}^n),$ $1<p \leq 2.$ (Hausdorff-Young Inequality).
\item $X=\Lambda^{s}(\mathbb{T}^1)=B^{s}_{\infty,\infty}(\mathbb{T}),$ $0<s< 1.$ In fact, by Lemma \ref{general}, if $p>(s+\frac{1}{2})^{-1}$ then $$\Vert \widehat{f}\Vert_{L^p}\lesssim \Vert f \Vert_{\Lambda^s}.$$
\end{itemize}
Hence
\begin{align}\sum_{\xi\in\mathbb{Z}^n} \Vert H_{\xi} \Vert^{r}_{B^{w_2}_{p_2, q_2}}\Vert G_{\xi} \Vert^{r}_{  (X )'}\lesssim \sum_{\xi\in\mathbb{Z}^n} C^r \langle \xi\rangle^{r(w_2+m+\delta(2k))}<\infty.  
\end{align}
As a consequence of this, we obtain the $r$-nuclearity of $T_{a}:X\rightarrow B^{w_2}_{p_2,q_2},$ where $0\leq w_{2}<2k-n,\,\, m<-\frac{n}{r}-w_2-\delta(2k),$  $1\leq p_2\leq \infty$ and $1\leq q_2\leq  \infty.$ 
\end{proof}
We end this section with the following result on $r$-nuclearity of periodic operators on H\"older spaces. 
\begin{corollary}
Let $0< r\leq 1 $, $0<s,w< 1,$ and $0\leq \rho,\delta\leq 1.$ Let us assume that $a\in S^{m}_{\rho,\delta,0,2}(\mathbb{T}\times \mathbb{Z}).$ If $m<-\frac{1}{r}-w-2\delta,$ then $T_a:\Lambda^{s}(\mathbb{T})\rightarrow \Lambda^{w}(\mathbb{T})$ is a $r$-nuclear operator.
\end{corollary}
\begin{proof}
Let us apply Theorem \ref{holder} with $X^{(1)}=B^{s}_{\infty,\infty}(\mathbb{T})=\Lambda^{s}(\mathbb{T}),$ $k=1,$ and $w_{2}=w.$
\end{proof}
\section{Trace formulae for $r$-nuclear pseudo-differential operators on Besov spaces}
In this section we provide trace formulae for $r$-nuclear operators on periodic Besov spaces.  
We recall the following result due to Grothendieck (see \cite{GRO}).
\begin{theorem}\label{gro}
Let $E$ be a Banach space and  $T:E\rightarrow E$ be a $\frac{2}{3}$-nuclear operator. Then the nuclear trace agrees with the sum of the eigenvalues $\lambda_{n}(T)$ of  $T,$ (with multiplicities counted).
\end{theorem}

Using this result we have our first Grothendieck-Lidskii trace formula for $r$-nuclear operators on periodic Besov spaces:
\begin{theorem}\label{tr}
Let $0<r\leq \frac{2}{3},$  $0<\alpha\leq \frac{1}{2},$ $0\leq \rho,\delta \leq 1,$ and  $k>\frac{n}{2},$ $k\in\mathbb{N}.$  Let us consider $a\in S^{m}_{\rho,\delta, 0,2k}(\mathbb{T}^n\times \mathbb{Z}^n).$ Under the following conditions,
\begin{itemize}
\item $0<w_{1}=\alpha \cdot n<2k-n, \,\,1< p_1\leq 2,$ $q_1=(\alpha+\frac{1}{p_1'})^{-1}.$
\item $m<-\frac{n}{r}-w_1-\delta(2k),$  
\end{itemize}
the pseudo-differential operator $T_a:B^{w_1}_{p_1,q_1}\rightarrow B^{w_1}_{p_1,q_1}$ is $r$-nuclear and the trace of $T,$  is given by
\begin{equation}
\mathrm{Tr}(T_a)=\sum_{\xi\in\mathbb{Z}^n}\int_{\mathbb{T}^n}a(x,\xi)dx=\sum_{n}\lambda_{n}(T_a)
\end{equation}
where $\lambda_{n}(T_a)$ is the sequence of eigenvalues of $T_{a}$ with multiplicities taken into account.
\end{theorem}

\begin{proof} We observe that by Theorem \ref{t1}, the operator $T_a$ is $r-$nuclear.  Let us denote by $\lambda_{n}(T_a)$ the sequence of eigenvalues of $T_{a}$ with multiplicities taken into account. Since $0<r\leq\frac{2}{3},$ from the  Theorem \ref{gro} we obtain
\begin{align*}
\sum_{n}\lambda_{n}(T_a)=&\text{Tr}(T_a)=\sum_{\xi\in\mathbb{Z}^n}G_{\xi}(H_\xi)\\
&=\sum_{\xi\in\mathbb{Z}^n}\widehat{H}_\xi(\xi)=\sum_{\xi\in\mathbb{Z}^n}\widehat{a}(0,\xi)=\sum_{\xi\in\mathbb{Z}^n}\int_{\mathbb{T}^n}a(x,\xi)dx.
\end{align*}
\end{proof}
\begin{corollary}\label{trcoroll}
Let $0<r\leq 1,$  $0<\alpha\leq \frac{1}{2},$ $0\leq \rho,\delta \leq 1,$ and  $k>\frac{n}{2},$ $k\in\mathbb{N}.$  Let us consider $a\in S^{m}_{\rho,\delta, 0,2k}(\mathbb{T}^n\times \mathbb{Z}^n).$ Under the following conditions,
\begin{itemize}
\item $0<w_{1}=\alpha \cdot n<2k-n, \,\,1< p_1\leq 2,$ $q_1=(\alpha+\frac{1}{p_1'})^{-1}.$
\item $m<-\frac{n}{r}-w_1-\delta(2k),$  
\end{itemize}
the pseudo-differential operator $T_a:B^{w_1}_{p_1,q_1}\rightarrow B^{w_1}_{p_1,q_1}$ is $r$-nuclear and the trace of $T_a,$  is given by
\begin{equation}\label{tracecor1}
\mathrm{Tr}(T_a)=\sum_{\xi\in\mathbb{Z}^n}\int_{\mathbb{T}^n}a(x,\xi)dx.
\end{equation}
\end{corollary}
\begin{proof}
By Theorem \ref{t1} $T_{a}$ is a $r$-nuclear operator. The trace formula \eqref{tracecor1} now follows from Theorem \ref{mainapp}.
\end{proof}
\noindent In order to prove our next trace formula, we use the following result by O. Reinov and Q. Latif, which extends the Grothendieck-Lidskii trace formula for $r \in ( \frac{2}{3}, 1]$  (see \cite{O}).
\begin{theorem}\label{lat} Let $Y$ be a subspace of an $L^{p}(\mu)$ space, $1\leq p\leq \infty$. Assume that $T$ is a $r$-nuclear operator on $Y,$ where $1/r=1+|1/2-1/p|.$ Then,  the (nuclear) trace of T is well defined, the sequence of eigenvalues $\lambda_{n}(T)$ of $T$ (with multiplicities counted) is summable and 
\begin{equation}\label{latifreinov}
\mathrm{Tr}(T)=\sum_{n}\lambda_n(T). 
\end{equation} 
\end{theorem}
Notice that as a consecuence of Theorem \ref{t1}, if $a\in S^{m}_{\rho,\delta, 0,2k}(\mathbb{T}^n\times \mathbb{Z}^n),$  $r= 1$, $1<p_{1}<2$, $\alpha=\frac{1}{p_1}- \frac{1}{2},$ and $w_1,m$ are index satisfying the conditions stated there,  
the pseudo-differential operator $T_a: B^{w_1}_{p_1,q_1}(\mathbb{T}^n)\rightarrow B^{w_2}_{p_2,q_2}(\mathbb{T}^n) $ is nuclear.
It follows from Theorem 5.2 of \cite{RuzBesov}  that $Y=B^{\omega_1}_{p_{1},2}$ is a subspace of $L^{2}(\mathbb{T}^n)$ . By applying Theorem \ref{lat} (with $p=2,$ $Y=B^{\omega_1}_{p_{1},2}$ and $r=1$), we conclude that \eqref{latifreinov} holds for $T=T_a.$ From the Theorem \ref{mainapp}, we have that
$$\text{Tr}(T_a)=\sum_{\xi\in\mathbb{Z}^n}\int_{\mathbb{T}^n}a(x,\xi)dx=\sum_{n}\lambda_{n}(T_a).$$
Thus, we summarize this facts in the following

\begin{theorem}
Let $1< p_1<2$ and $\alpha=\frac{1}{p_1}-\frac{1}{2}.$ Let  $0\leq \rho,\delta \leq 1,$ and  $k>\frac{n}{2},$ $k\in\mathbb{N}.$  Let us consider $a\in S^{m}_{\rho,\delta, 0,2k}(\mathbb{T}^n\times \mathbb{Z}^n).$ Under the following conditions,
\begin{itemize}
\item $0<w_{1}=\alpha \cdot n<2k-n.$ 
\item $m<-{n}-w_1-\delta(2k),$  
\end{itemize}
The operator $T_{a}$ is  nuclear on $B^{\omega_1}_{p_{1},2}$ and 
$$\mathrm{Tr}(T_a)=\sum_{\xi\in\mathbb{Z}^n}\int_{\mathbb{T}^n}a(x,\xi)dx=\sum_{n}\lambda_{n}(T_a).$$
The sequence $\lambda_{n}(T_a)$ is conformed by the  eigenvalues of $T_a$  with multiplicities counted.
\end{theorem}

\section{Trace formulae for Fourier multipliers on the torus}

In this section we provide trace formulae for $r$-nuclear Fourier multipliers on periodic Besov spaces.  We denote by $S^{m}_{0}(\mathbb{T}^n\times \mathbb{Z}^n)$ the set of functions $a:\mathbb{Z}^{n}\rightarrow \mathbb{C}$ satisfying $|a(\xi)|\leq C\langle\xi\rangle^m.$

\begin{theorem}
Let $0<r\leq \frac{2}{3},$ and let $0<\alpha\leq \frac{1}{2}.$  Let us consider $a(\xi)\in S^{m}_{0}(\mathbb{T}^n\times \mathbb{Z}^n).$ Under the following conditions,
\begin{itemize}
\item $1< p_1\leq 2,$ $q_1=(\alpha+\frac{1}{p_1'})^{-1},$ 
\item $m<-\frac{n}{r}-\alpha\cdot n,$  
\end{itemize}
the Fourier multiplier $T_a:B^{w}_{p_1,q_1}\rightarrow B^{w}_{p_1,q_1}$ is $r$-nuclear for every $w\in\mathbb{R}$ and the trace of $T,$  is given by
\begin{equation}
\mathrm{Tr}(T_a)=\sum_{\xi\in\mathbb{Z}^n}a(\xi)=\sum_{n}\lambda_{n}(T_a)
\end{equation}
where $\lambda_{n}(T_a)$ is the sequence of eigenvalues of $T_{a}$ with multiplicities taken into account.
\end{theorem}
\begin{proof} We observe that by Theorem \ref{t1}, the operator $T_a$ is $r-$nuclear on $B^{\alpha\cdot n}_{p_1,q_1}(\mathbb{T}^n)$ with $0<\alpha\cdot n<2k-n$.  By using Theorem \ref{equivalencia} we extend this result to every $w\in\mathbb{R}$. Now, if we denote by $\lambda_{n}(T_a)$ the sequence of eigenvalues of $T_{a}$ with multiplicities taken into account and considering $0<r\leq\frac{2}{3},$ from the  Theorem \ref{gro} we obtain
\begin{align*}
\sum_{n}\lambda_{n}(T_a)=&\text{Tr}(T_a)=\sum_{\xi\in\mathbb{Z}^n}G_{\xi}(H_\xi)\\
&=\sum_{\xi\in\mathbb{Z}^n}\widehat{H}_\xi(\xi)=\sum_{\xi\in\mathbb{Z}^n}\widehat{a}(0,\xi)=\sum_{\xi\in\mathbb{Z}^n}	a(\xi).
\end{align*}
\end{proof}
An immediate consequence of the Theorem above is the following.
\begin{corollary}
Let $0<r\leq 1,$  and let $0<\alpha\leq \frac{1}{2}.$  Let us consider $a(\xi)\in S^{m}_{0}(\mathbb{T}^n\times \mathbb{Z}^n).$ Under the following conditions,
\begin{itemize}
\item $1< p_1\leq 2,$ $q_1=(\alpha+\frac{1}{p_1'})^{-1}.$
\item $m<-\frac{n}{r}-\alpha\cdot n,$  
\end{itemize}
the Fourier multiplier $T_a:B^{w}_{p_1,q_1}\rightarrow B^{w}_{p_1,q_1}$ is $r$-nuclear for every $w\in\mathbb{R}$ and the nuclear trace of $T_a,$  is given by
\begin{equation}
\mathrm{Tr}(T_a)=\sum_{\xi\in\mathbb{Z}^n}a(\xi).
\end{equation}
\end{corollary}
\begin{proof}
By Theorem \ref{t1} $T_{a}$ is a $r$-nuclear operator. The trace formula \eqref{tracecor1} now follows from Theorem \ref{mainapp}.
\end{proof}
Now we present the following result which can be proved by using similar arguments as above. 
\begin{theorem}
Let $1< p_1<2$ and $\alpha=\frac{1}{p_1}-\frac{1}{2}.$   Let us consider $a(\xi)\in S^{m}_{0}(\mathbb{T}^n\times \mathbb{Z}^n).$ If
 $m<-{n}-\alpha\cdot n,$  the Fourier multiplier $T_{a}$ is a  nuclear operator on $B^{w}_{p_{1},2}$ for every $w\in\mathbb{R}$ and 
$$\mathrm{Tr}(T_a)=\sum_{\xi\in\mathbb{Z}^n}a(\xi)=\sum_{n}\lambda_{n}(T_a).$$
The sequence $\lambda_{n}(T_a)$ is conformed by the  eigenvalues of $T_a$  with multiplicities counted.
\end{theorem}
\begin{remark}
Now, we discuss  our main results in the periodic case. Theorem \ref{t1}, if we consider smooth symbols (i.e with derivatives of arbitrary order), we obtain the $r$-nuclearity in Besov spaces of pseudo-differential on the torus $\mathbb{T}^n,$ associated to symbols of order less that $-\frac{n}{r},$ and some conditions of the parameters $p_{i},q_i$ and on $w_i.$ This is a  expected fact, in analogy with some results by Ruzhansky and Delgado in $L^{p}$ spaces (c.f. \cite{DR, DR-1, DR1, DR3}). The conclusion above is same for Theorem \ref{t2222}. Also, it is important to mention that trace formulae obtained in the last two sections are versions in Besov spaces of ones obtained by Delgado and Wong in $L^p$ spaces \cite{DW}. 
\end{remark}

We end this section with the following examples where, by using results above we compute the trace of the heat kernel and suitable powers of the Bessel potential on periodic Besov spaces. 
\begin{example}
 Let $\mathcal{L}_{\mathbb{T}^n}$ be the Laplacian on the torus $\mathbb{T}^n,$  for every $s\in \mathbb{R},$ the Bessel potential of order $s$ denoted by $(I-\mathcal{L}_{\mathbb{T}^n})^{s}$ is the periodic operator with symbol $a_s(x,\xi)=\langle \xi \rangle^{s}.$ If $0<r\leq 1,$ and  $0<\alpha\leq \frac{1}{2}$, by using Corollary \ref{trcoroll}, under the following conditions,
\begin{itemize}
\item $\alpha>0, \,\,1< p_1\leq 2,$ $q_1=(\alpha+\frac{1}{p_1'})^{-1}$, and  $m<-\frac{n}{r}-\alpha\cdot n,$  
\end{itemize}
the  operator $(I-\mathcal{L}_{\mathbb{T}^n})^{-\frac{m}{2} }$ with symbol $a(\xi)=\langle \xi \rangle^{-m}\in S^{-m}_{1,0}(\mathbb{T}^n\times \mathbb{Z}^n)$  is $r$-nuclear on every  $B^{w_1}_{p_1,q_1}$ and its trace is given by
\begin{equation}
\mathrm{Tr}((I-\mathcal{L}_{\mathbb{T}^n})^{-\frac{m}{2} })=\sum_{\xi\in\mathbb{Z}^n}\langle \xi\rangle^{-m}.
\end{equation}
\end{example}
\begin{example}
If $t>0,$ the heat kernel $e^{-t\mathcal{L}_{\mathbb{T}^n}   }$ is the operator with symbol $a_t(x,\xi)=e^{-t|\xi|^{2}}\in S^{-\infty}_{1,0}(\mathbb{T}^n\times \mathbb{Z}^n).$ Newly, by Corollary \ref{trcoroll}, if $w_1,p_1$ and $q_{1}$ satisfy the condition above, $e^{-t\mathcal{L}_{\mathbb{T}^n}   }$ is a $r-$nuclear operator on $B^{w_1}_{p_1,q_1}$ and its trace is given by
\begin{equation}
\mathrm{Tr}(e^{-t\mathcal{L}_{\mathbb{T}^n}   })=\sum_{\xi\in\mathbb{Z}^n}e^{-t|\xi|^{2}}.
\end{equation}
\end{example}

\noindent \textbf{Acknowledgments.} 
The author is indebted with  Alexander Cardona for helpful comments on an earlier draft of this paper.
The author would like to warmly  thank the anonymous referee for his  remarks and important advices leading to several improvements of the original paper.
\bibliographystyle{amsplain}

\end{document}